\documentclass[12pt]{amsart}
\usepackage{amsfonts}
\usepackage{amsmath}
\usepackage[mathscr]{eucal}
\usepackage{eurosym}
\usepackage{amssymb}
\usepackage{amstext}
\usepackage{amsthm}
\usepackage{mathrsfs}
\usepackage{centernot}
\usepackage{enumitem}
\usepackage{graphicx, color}
\usepackage{caption}
\usepackage{subcaption}
\usepackage[a4paper, total={6in, 8in}]{geometry}

\newtheorem{thm}{Theorem}[section]

\newtheorem{lemma}[thm]{Lemma}

\theoremstyle{definition}

\theoremstyle{remark}
\newtheorem{remark}[thm]{Remark}
\theoremstyle{example}
\newtheorem{example}[thm]{Example}

\begin{document}
\title[On the Hurwitz-type zeta function]{On the Hurwitz-type zeta function associated to the Lucas sequence}
\author[L. Smajlovi\'{c}]{Lejla Smajlovi\'{c}}
\address{School of Economics and Business\\
University of Sarajevo\\
Trg Oslobodjenja Alija Izetbegovi\'{c} 1\\
71 000 Sarajevo\\
Bosnia and Herzegovina}
\email{lejla.smajlovic@efsa.unsa.ba}
\author[Z. \v{S}abanac]{Zenan \v{S}abanac}
\address{Department of Mathematics\\
University of Sarajevo\\
Zmaja od Bosne 35\\
71 000 Sarajevo\\
Bosnia and Herzegovina}
\email{zsabanac@pmf.unsa.ba}
\author[L. \v{S}\'{c}eta]{Lamija \v{S}\'{c}eta}
\address{School of Economics and Business\\
University of Sarajevo\\
Trg Oslobodjenja Alija Izetbegovi\'{c} 1\\
71 000 Sarajevo\\
Bosnia and Herzegovina}
\email{lamija.sceta@efsa.unsa.ba}

\begin{abstract}
We study the theta function and the Hurwitz-type zeta function associated to the Lucas sequence $U=\{U_n(P,Q)\}_{n\geq 0}$ of the first kind determined by the real numbers $P,Q$ under certain natural assumptions on $P$ and $Q$. We deduce an asymptotic expansion of the theta function $\theta_U(t)$ as $t\downarrow 0$ and use it to obtain a meromorphic continuation of the Hurwitz-type zeta function $\zeta _{U}\left( s,z\right) =\sum\limits_{n=0}^{\infty }\left(z+U_{n}\right) ^{-s}$ to the whole complex $s-$plane. Moreover, we identify the residues of $\zeta _{U}\left( s,z\right)$ at all poles in the half-plane $\Re(s)\leq 0$.
\end{abstract}

\maketitle


\section{Introduction}

The Lucas sequences $\{U_n(P,Q)\}_{n\geq 0}$ and $\{V_n(P,Q)\}_{n\geq 0}$ of the first and the second kind associated to arbitrary real numbers $P,Q$ are defined for all non-negative integers $n$ by
\begin{equation}\label{eq: Lucas seq def}
U_{n}=\frac{a^{n}-b^{n}}{a -b }\quad \text{and} \quad V_{n}=a^{n}+b^{n},
\end{equation}
where $a$ and $b$ ($a\geq b$) are the roots of the quadratic equation $x^{2}-Px+Q =0$ and by definition $U_0=1$ and $V_0=2$. Those sequences were first studied by E. Lucas in \cite{Luacas}, who considered the special case when $P,Q$ are relatively prime integers.

Some well-known special cases of sequences  $\{U_n(P,Q)\}_{n\geq 0}$ and $\{V_n(P,Q)\}_{n\geq 0}$ include Fibonacci numbers $F_n=U_{n}(1,-1)$, Lucas numbers $L_n=V_{n}(1,-1)$, Pell numbers $U_{n}(2,-1)$, Pell-Lucas numbers $V_{n}(2,-1)$, Jacobsthal numbers $U_{n}(1,-2)$ and Jacobsthal-Lucas numbers $V_{n}(1,-2)$, and they are all indexed in the On-Line Encyclopedia of Integer Sequences \cite{oeis} as $A000045$, $A000032$, $ A000129$, $A002203$, $A001045$ and $A014551$, respectively. There are many further generalizations of Lucas sequences; some of them are given in \cite{Bil, HST, Omer}, as well as numerous articles studying different properties of those sequences (see e.g. \cite{HSZ, kilic, OP2, OP1, sanna, SZ}).

Zeta and $L-$functions in number theory carry important information related to the object they are associated to. Hence, it is of interest to study zeta-type functions associated to Lucas sequences. The zeta function $\zeta_{\{F_n\}}(s):= \sum_{n\geq 1} F_n^{-s}$, associated to the Fibonacci sequence was studied in \cite{Egami} and \cite{Navas}, where a meromorphic continuation of $\zeta_{\{F_n\}}(s)$ to $s\in\mathbb{C}$ was deduced. Moreover, in \cite{Navas} some interesting algebraicity and transcendence results for special values of $\zeta_{\{F_n\}}(s)$ were derived, see also \cite{MurtyFib} for related results. Analytic continuation of the multiple Fibonacci zeta functions was studied in \cite{RoMe}.

The Lucas zeta function
\begin{equation} \label{eq: defn Lucas zeta f-on}
\zeta_{\{U_n(P,Q)\}}(s)=\zeta _{U}\left( s\right) :=\sum_{n=1}^{\infty }\frac{1}{U_{n}\left(
P,Q\right) ^{s}},\quad  s\in \mathbb{C}, \Re(s) >0,
\end{equation}
associated to the Lucas sequence $U=\{U_n(P,Q)\}_{n \geq 0}$ of the first kind was studied by K. Kamano in \cite{Kamano}, who  described in detail its polar structure. Kamano also studied the $L-$function defined by twisting $\zeta_{U}(s)$ by a primitive, multiplicative Dirichlet character. Going further, N. K. Meher and S. S. Rout \cite{MeRo} extended results from \cite{Kamano} and proved analytic continuation of the multiple Lucas zeta function.

Hurwitz-type zeta functions, a natural generalization of zeta and $L-$functions associated to certain sequences appear in many mathematical and physical disciplines. The Hurwitz-type zeta function associated to the sequence $A=\{a_n\}_{n\geq 0}$ of complex numbers is formally defined for $\Re (s)\gg 0$ by
$$
\zeta_A(s,z)=\sum_{n=0}^{\infty }\frac{1}{\left(
z+a_{n}\right) ^{s}},
$$
for all $z\in\mathbb{C}$ such that $z+a_n \notin \mathbb{R}_{\leq 0}$ for all $n\geq 0$ and $\left(z+a_{n}\right) ^{-s}$ is defined using the principal branch of the logarithm.

When the sequence $A$ is the sequence of eigenvalues of a certain operator, Hurwitz-type zeta functions appears in the literature under the name "two parameter spectral function", or "generalized zeta function", see, e.g. \cite{Vor2} and \cite{GJ}, or \cite[Chap. 4]{Elizalde} where the study of Epstein-Hurwitz zeta functions is nicely presented.

A thorough study of Hurwitz-type zeta functions associated to a general sequence $A$ of complex numbers, satisfying certain mild conditions was conducted in \cite{JOrgenson i Lang}, where generalizations of various were derived (e.g. the Lerch formula, the Gauss formula, the Stirling formula and many others). A more general study is presented in \cite{Illies}, where the sequence $A$ is replaced by a countable set $D$, the so-called "directed divisor".

When the sequence $A$ is the sequence of zeros of a zeta or $L-$function, the Hurwitz-type zeta function appears under the name "super-zeta" function. Polar structure of such functions can be described using the approach explained in \cite{Vor1} (see also \cite{FJS} for a result relating super-zeta functions to determinants of Laplacians on cofinite Fuchsian groups).

In this paper, we study the Hurwitz-type zeta function associated to the sequence $U=\{U_n(P,Q)\}_{n\geq 0}$ of Lucas numbers of the first kind associated to real parameters $P,Q$ which are arbitrary, but fixed throughout the paper (and hence, omitted from notation) and which satisfy the technical assumption \eqref{assumption} below. This function is defined for $\Re (s)>0$ and $\Re (z)>0$ by
\begin{equation} \label{eq: zeta defn}
\zeta _{U}\left( s,z\right) =\sum_{n=0}^{\infty }\frac{1}{\left(
z+U_{n}\right) ^{s}},
\end{equation}%
where, as is common, we use the principal branch of the logarithm to define $(z+U_n)^{-s}$.

In our main theorem (Theorem \ref{thm: main} below) we prove that the Hurwitz-type zeta function $\zeta _{U}\left( s,z\right)$ admits a meromorphic continuation to the whole complex $s-$plane for $z\in \mathbb{C}\setminus
\left( -\infty ,0\right] $, and we give the complete list of possible poles of this meromorphic function and their corresponding residues. It is interesting to notice that for all sequences $U=\{U_n\}_{n\geq 0}$, the zeta function $\zeta _{U}\left( s,z\right)$ possesses poles at non-positive integers and other poles located on certain vertical lines in the half-plane $\Re(s)\leq 0$.

In order to illustrate our results, we study the Hurwitz-type zeta function associated to the sequence  $U=\{U_n\}_{n\geq 0}$, where $U_0=1$ and $U_n=\sum_{j=0}^{n-1}a^j$, for $n\geq 1$ is the partial sum of the divergent geometric series $\sum_{n\geq 0} a^n$, $a>1$ and prove that the residue of $\zeta _{U}\left( s,z\right)$ at $-\ell$, $\ell\in\mathbb{N}_0$ equals $\frac{1}{\log a}\left(z+\frac{1}{1-a}\right)^{\ell}$.

Our second example is devoted to the Hurwitz-type zeta function associated to the Fibonacci sequence $\{F_n\}_{n\geq 0}$. Let $\varphi=(\sqrt{5}+1)/2$ be the golden ratio and  $\bar\varphi=(1-\sqrt{5})/2$ its conjugate. We prove that the residue of the Hurwitz-type zeta function $\zeta_{\{F_n\}}(s,z)$ at the pole $-\ell$, $\ell\in\mathbb{N}_0$ equals $1/\log\varphi$ times sum of the terms in the trinomial $ (\sqrt{\varphi/5} + \sqrt{|\bar\varphi|/5}+ z)^{\ell}$ which possess rational coefficients.

The approach we undertake in the study of $\zeta _{U}\left( s,z\right) $ is different from previous studies \cite{Egami, MurtyFib, Navas} on this topic (and which is based on the Taylor series expansion), due to the fact that the sequence $\{U_n+z\}_{n\geq 0}$, for  $z\in \mathbb{C}\setminus \left( -\infty ,0\right] $  does not satisfy the Binet-type formula. Namely, our starting point is the fact that $\Gamma(s)\zeta _{U}\left( s,z\right)$, for $\Re (s)$ and $\Re (z)$ large enough equals the Laplace-Mellin transform of the corresponding theta function
\begin{equation} \label{eq: theta def}
\theta _{U}\left( t\right) :=\sum\limits_{n=0}^{\infty }e^{-U_{n}t}\text{, }%
t\in\mathbb{R}_{> 0}.
\end{equation}
This approach can be traced back to Riemann and is nicely explained in \cite{JOrgenson i Lang}. However, results of \cite{JOrgenson i Lang} on meromorphic continuation of the Hurwitz-type zeta function $\zeta_A(s,z)$ associated to a general sequence $A$ of complex numbers can not be applied in our setting, due to the fact that the Lucas zeta function has infinitely many poles in every strip in the half-plane $\Re (s)<0$ of width greater than one. Though more general than \cite{JOrgenson i Lang}, results of \cite{Illies} can not be applied for the same reason.

To overcome the above-mentioned problems, using the Mellin inversion, in Theorem \ref{thm: theta exp} below, we derive an asymptotic expansion for the theta function $\theta _{U}\left( t\right)$ as $t\downarrow 0$, up to $O(t^m)$, for an arbitrary positive integer $m$. Using this expansion, we are able to deduce the meromorphic continuation and the polar structure of $\zeta _{U}\left( s,z\right)$.

The structure of the paper is the following: In Section 2 we prove certain properties of the Lucas zeta function, needed to derive an asymptotic expansion for the theta function $\theta _{U}\left( t\right)$ in Section 3. Section 4 is devoted to description of the polar structure of the function $\zeta _{U}\left( s,z\right)$. In the last section we derive an explicit evaluation of the polar structure of Hurwitz-type zeta functions associated to the Fibonacci sequence and to the sequence of partial sums of a divergent geometric series. We end the paper with concluding remarks, where we discuss the case when $Q=0$ and future projects related to zeta functions of recurrence sequences of the third order.

\section{Properties of the Lucas zeta function}

Let $P$ and $Q$ be arbitrary real numbers such that
\begin{equation}
P>0,\quad Q\neq 0\quad \text{and}\quad \left\{
\begin{array}{ll}
Q\leq P-1, & \hbox{$P>2$,} \\
Q<P-1, & \hbox{$0<P\le 2$.}%
\end{array}%
\right.  \label{assumption}
\end{equation}
In the sequel we will assume that $P,Q \in \mathbb{R}$ are arbitrary, fixed numbers satisfying \eqref{assumption} and we will omit them from the notation. The following lemma summarizes properties of the sequence $\{U_{n}\}_{n \geq 0}$.

\begin{lemma} \label{DIR}
The Lucas sequence $\{U_{n}\}_{n \geq 0}$ and the corresponding Lucas zeta function $\zeta _{U}$ possess the following properties:

\begin{itemize}
\item[(i)] $a>1$, $a>|b|>0$ and $U_{n}>0$ for all integers $n\geq 1$.

\item[(ii)] The infinite series $\zeta _{U}\left( s\right) $ is absolutely
convergent for $\Re (s)>0$.

\item[(iii)] For every $C\in \mathbb{R}_{\geq 0}$, there exists a positive integer $n_C$ such that the inequality $U_{n} > C$ holds true for all integers  $n\geq n_C$.
\end{itemize}
\end{lemma}

\begin{proof}
Properties (i) and (ii) follow from \cite[Prop. 3]{Kamano}. Part (iii) follows from the fact that $\underset{n\rightarrow +\infty }{\lim }\frac{a^{n}-b^{n}}{a-b}=+\infty$.
\end{proof}

Under assumptions \eqref{assumption} on $P,Q$, in \cite[Thm. 4 on p. 640]{Kamano} it was proved that the Lucas zeta function $\zeta _{U}\left( s\right) $, associated to the Lucas sequence $\{U_{n}\}_{n \geq 0}$ of the first kind can be meromorphically continued to the whole $s$-plane and written as
\begin{equation} \label{eq: zeta cont Kamano}
\zeta _{U}\left( s\right)=D^{s/2}\sum_{k=0}^{\infty }\binom{-s}{k}%
\left( -1\right) ^{k}\frac{Q^{k}}{a^{s+2k}-Q^{k}} =D^{s/2}\sum_{k=0}^{\infty }\binom{-s}{k}%
\left( -1\right) ^{k}\frac{a^{-s-k}b^k}{1-a^{-s-k}b^{k}},
\end{equation}%
where $D=P^{2}-4Q=\left( a-b\right) ^{2}$. Moreover, it was proved that  $\zeta _{U}\left( s\right) $ is holomorphic except for possible simple poles at
\begin{equation}\label{eq: s k,n defn}
s=s_{k,n}=-2k+\frac{k\log |Q|}{\log a}+\frac{\left( 2n+l_{Q,k}\right) \pi i}{
\log a},
\end{equation}
where $k\geq 0$ and $n$ are arbitrary integers and $l_{Q,k}:=\left\{
\begin{array}{ll}
k, & \hbox{$Q<0$,} \\
0, & \hbox{$Q>0$.}%
\end{array}%
\right. $

The residue of $\zeta _{U}\left( s\right) $ at $s=s_{k,n}$ is given by
\begin{equation*}
\mathrm{Res}_{s=s_{k,n}}\left( \zeta _{U}\right) =D^{\frac{s_{k,n}}{2}}%
\binom{-s_{k,n}}{k}\frac{\left( -1\right) ^{k}}{\log a}.
\end{equation*}

If for some $k\geq 1$ and $n\in \mathbb{Z}$ we have $-s_{k,n} \in \{0,1,\ldots,k-1\}$, then $\binom{-s_{k,n}}{k}=0$ and $s_{k,n}$ is not a pole of $\zeta _{U}\left( s\right) $. A simple computation shows that $\zeta_U(s)$ possesses a simple pole at $s=0$ with the Laurent series expansion given by
\begin{equation}\label{eq: series for zetaU at s=0}
  \zeta _{U}\left( s\right) =\frac{1}{s\log a}+\frac{1}{2}\left( \frac{\log D}{\log a}-1\right) +O\left( s\right), \text{  as  } s\to 0.
\end{equation}

In the lemma below we summarize the cases at which $\zeta _{U}\left( s\right)$ possesses poles at negative integers.

\begin{lemma}\label{lem: poles at neg int}
  With the notation as above, under assumption \eqref{assumption}, the Lucas zeta function $\zeta _{U}\left( s\right) $ possesses poles at $s=-\ell,$ for some $\ell\in\mathbb{N}$ if and only if:
  \begin{itemize}
    \item[(i)] $Q=1$ and $\ell=2k$ is an even positive integer;
    \item[(ii)] $Q=-1$ and $\ell=4k$ is a positive integer divisible by $4$;
    \item [(iii)] $Q>0$, $b=a^{-p/q}$ for a non-negative rational number $p/q$ and
$$
\ell =\left\{
\begin{array}{l}
k_{1}(1+p/q)=-s_{k,0}\text{, for some }k=qk_{1}\in \mathbb{N}\text{, when }%
p\neq 0, \\
k\text{, }k\in \mathbb{N}\text{, when }p=0;%
\end{array}%
\right.
$$
    \item [(iv)] $Q<0$,  $b=-a^{-p/q}$ for a non-negative rational number $p/q$ and
$$
\ell =\left\{
\begin{array}{l}
2k_{1}(1+p/q)=-s_{k,-k/2}\text{, for some }k=2qk_{1}\in \mathbb{N}\text{,
when }p\neq 0, \\
2k,\text{ }k\in \mathbb{N}\text{, when }p=0.%
\end{array}%
\right.
$$
  \end{itemize}
In each of the four cases listed above the constant term the Laurent series expansion of $\zeta _{U}\left( s\right) $ at  $s=s_{k,0}=-\ell$ (for $Q>0$) or $s=s_{k,-k/2}=-\ell$ (for $Q<0$) is given by:
\begin{multline}\label{eq: const term}
  \mathrm{CT}_{s=-\ell}\zeta_U(s)= D^{-\ell/2}\binom{\ell}{k}(-1)^k\left( \frac{\log D - \log a}{2\log a} + \frac{H_{\ell-k} - H_{\ell}}{\log a}\right) +\\ + D^{-\ell/2}\sum_{j=0, j\neq k}^{\ell}\binom{\ell}{j}(-1)^j\frac{a^{\ell-j} b^j}{1-a^{\ell-j} b^j},
\end{multline}
where $H_n:=\sum_{j=1}^{n} \frac{1}{j}$ denotes the $n$th harmonic number and $H_0:=0$.
   \end{lemma}
\begin{proof}
Parts (i) and (ii) follow from \cite[Prop. 6]{Kamano}. Let us prove (iii). In order that $-\ell$ is a pole of $\zeta _{U}\left( s\right) $, there must exist $n\in \mathbb{Z}$ and $k\geq 0$ such that $- \ell=s_{n,k}$. For $Q>0$ (and hence $b>0$) this is fulfilled if and only if $n=0$ and
$$
\ell=k\left(1-\frac{\log b}{\log a}\right).
$$
Therefore, $-\ell$ is a pole of $\zeta _{U}\left( s\right) $ if and only if $\frac{\log b}{\log a}$ is a rational number and
$$k\left(1-\frac{\log b}{\log a}\right) \notin \{0,1,\ldots,k-1\}.$$
If $b=a^r$ for some positive rational number $r$ (obviously, $r<1,$ since $a>b)$ then for all positive integers $k$ such that $k(1-r)$ is a positive integer, this integer belongs to the set $\{0,1,\ldots,k-1\}$ and hence $-k(1-r)$ is not a pole of $\zeta_U(s)$. Therefore, $-\ell$ is a pole of $\zeta _{U}\left( s\right) $ if and only if $b=a^{-p/q}$ for some non-negative rational number $p/q$ (we assume $p,q$ to be relatively prime integers or $p=0$). When $p=0$, then $\ell=k$, for $k\in \mathbb{N}$. When $p/q>0$, obviously $\ell=k(1+p/q)$ and this is an integer if and only if $k\equiv 0 (\mathrm{mod}\ q)$. This proves part (iii).

The proof of part (iv) is similar. When $Q<0$, we have $l_{Q,k}=k$ and hence $s_{k,n}$ can be a negative integer only if $k=2m$ is an even positive integer and $n=-m$. In this case $b<0$ and
$$
-s_{2m,-m}=2m\left(1-\frac{\log |b|}{\log a}\right).
$$
reasoning as above completes the proof of (iv).

Now, it is left to prove that the constant term in all situations (i)--(iv) is given by \eqref{eq: const term}. We prove this for $Q>0$, $b=a^{-p/q}$ for a rational number $p/q\geq 0$ and $\ell=k(p+q)/q=-s_{k,0}$ for some $k=qk_1\in \mathbb{N}$. For such $k$ and $\ell$ we have, as $s\to -\ell$:
\begin{equation} \label{eq: fraction exp}
\frac{a^{-s-k}b^k}{1-a^{-s-k}b^{k}}=\frac{1-(s+\ell)\log a + O((s+\ell)^2)}{(s+\ell)\log a \left( 1-\tfrac{1}{2}(s+\ell)\log a  + O(s+\ell)^2\right)},
\end{equation}
\begin{equation}\label{eq: D exp}
  D^{s/2}= D^{-\ell/2} \left(1+\tfrac{1}{2}(s+\ell)\log D  + O((s+\ell)^2)\right)
\end{equation}
and
\begin{eqnarray} \label{eq: binomial exp}
  \binom{-s}{k} &=& \frac{\Gamma(-(s+\ell)+\ell +1)}{k! \Gamma(-(s+\ell) +\ell +1-k)} \\ \nonumber
   &=& \binom{\ell}{k}\left(1+(s+\ell)(\psi(\ell+1-k) - \psi(\ell+1)) + O((s+\ell)^2) \right),
\end{eqnarray}
where $\psi(x)$ denotes the digamma function. By letting $s\to -\ell$ we get
\begin{multline*}
\lim_{s\to -\ell}\left(\zeta_U(s) -D^{-\ell/2}\binom{\ell}{k}\frac{\left( -1\right) ^{k}}{(s+\ell)\log a}\right)=  D^{-\ell/2}\sum_{j=0, j\neq k}^{\ell}\binom{\ell}{j}(-1)^j\frac{a^{\ell-j} b^j}{1-a^{\ell-j} b^j} +\\+ \lim_{s\to -\ell}\left( D^{s/2}\binom{-s}{k} (-1)^k \frac{a^{-s-k}b^k}{1-a^{-s-k}b^{k}}-D^{-\ell/2}\binom{\ell}{k}\frac{\left( -1\right) ^{k}}{(s+\ell)\log a}\right).
\end{multline*}
Therefore, it is left to compute the limit on the right-hand side of the above equation. By multiplying \eqref{eq: fraction exp}, \eqref{eq: D exp} and \eqref{eq: binomial exp} we deduce that this limit equals
$$
D^{-\ell/2}\binom{\ell}{k}\frac{(-1)^k}{\log a}\left( \frac{\log D - \log a}{2} + \left(\psi(\ell+1-k) - \psi(\ell+1)\right)\right).
$$
Using the functional equation $\psi(x+1)=\psi(x)+1/x$ for the digamma function and the fact that $\psi(1)=-\gamma$, where $\gamma$ is the Euler constant, we complete the proof of \eqref{eq: const term} in case (iii). The proof in all other cases is similar, so we omit it.
\end{proof}

\begin{example}\rm
The zeta function associated to Jacobstahl numbers $J_n:=U_n(1,-2)$ is such that $Q<0$, $a=2$, $b=-1$, hence $\log |b|=0$ and the case (iv) of the above lemma with $p/q=0$ applies to deduce that the Jacobstahl zeta function $\zeta_{\{J_n\}}(s):=\sum_{n=1}^{\infty} J_n^{-s}$, initially defined for $\Re(s)>0$ possesses meromorphic continuation to the whole complex plane with poles at points $s=-k+(2n+k)\pi i /\log 2$, $k\geq 0$, $n\in\mathbb{Z}$. Specially, this function possesses poles at all non-positive even integers $-2\ell$, $\ell\in\mathbb{N}_0$.
\end{example}

\section{The theta function associated to the Lucas sequence}

In this section we derive properties of the theta function $\theta_U(t)$ associated to the Lucas sequence of the first kind, and defined by \eqref{eq: theta def}. The following lemma describes the asymptotic behavior of $\theta_U(t)$ for large and small values of $t$ and shows that the series \eqref{eq: theta def} converges uniformly in $t$ on every set $[t_0,\infty)$, for $t_0 >0$.

\begin{lemma}
\label{lem2} The theta function $\theta
_{U}\left( t\right) $ possesses the following properties:

\begin{itemize}
\item[(i)] For given numbers $C,t_{0}\in \mathbb{R}_{\geq 0}$, there
exist $N\in \mathbb{N}$ and $K\in \mathbb{R}_{\geq 0}$ such that $\left\vert
\sum\limits_{n=N}^{\infty }e^{-U_{n}t}\right\vert \leq Ke^{-Ct}\quad \text{%
for}\quad t\geq t_{0}$ . 

\item[(ii)] For $\delta\in(0,1)$, there exist $\alpha ,C\in \mathbb{R}_{\geq 0}$ such that the inequality $\left\vert
\sum\limits_{n=N}^{\infty }e^{-U_{n}t}\right\vert \leq \frac{C}{t^{\alpha }}$ holds true for all $N\in \mathbb{N}$ and $t\in \left( 0,\delta \right] $.
\end{itemize}
\end{lemma}

\begin{proof}
Proof is similar to proof of \cite[Thm. 1.12]{JOrgenson i Lang}, so we give only a sketch here.

Part (i) follows from Lemma \ref{DIR} by taking $N=n_C$ and $K= \sum\limits_{n=n_C}^{\infty }e^{-(U_{n}-C)t_0}$.

Part (ii) follows from Lemma \ref{DIR} combined with the inequality $x^{\beta}e^{-x}\leq c$ which holds true for any $x\geq 0$, $\beta>0$, with a constant $c$ depending only upon $\beta$.
\end{proof}

In view of Lemma \ref{lem: poles at neg int}, for an arbitrary, fixed $k\in\mathbb{N}_0$ and $a,b$ as above, we define the set $A_{a,b}(k)$ by setting
$$
A_{a,b}(k)=\left\{
           \begin{array}{ll}
             \mathbb{Z}, & \text{ if  } \frac{\log|b|}{\log a}k\notin \mathbb{Z}_{\leq 0} \text{ or } \left(Q=ab <0 \text{ and } k \text{ is odd}\right); \\
              \mathbb{Z}\setminus\{0\}, & \text{ if  } \frac{\log|b|}{\log a}k\in \mathbb{Z}_{\leq 0} \text{ and } Q=ab>0; \\
             \mathbb{Z}\setminus\{-k/2\}, & \text{ if  } \frac{\log|b|}{\log a}k\in \mathbb{Z}_{\leq 0} \text{ and } Q=ab<0 \text{ and } k \text{ is even}.
           \end{array}
         \right.
$$
For a positive integer $m$ we also define the subset $B(m,a,b)$ of the set $\{1,2,\ldots,2m\}$ to be the set of numbers $\ell\in \{1,2,\ldots,2m\}$ for which there exists $j\in\mathbb{N}$ such that
$$
\ell=\left\{
           \begin{array}{ll}
            j\left(1-\frac{\log|b|}{\log a}\right) , & \text{ if  } \frac{\log|b|}{\log a}j\in \mathbb{Z}_{\leq 0} \text{ and } Q=ab>0; \\
             j\left(1-\frac{\log|b|}{\log a}\right) , & \text{ if  } \frac{\log|b|}{\log a}j\in \mathbb{Z}_{\leq 0} \text{ and } \text{ } Q=ab<0 \text{ and } j \text{ is even}.
           \end{array}
         \right.
$$
If the value of $j$ described above does not exist, then, by definition $B(m,a,b)=\emptyset$. With this notation, we can state and prove our first main result in which we deduce the asymptotic expansion of $\theta_U(t)$ as $t\downarrow 0$.

\begin{thm} \ \label{thm: theta exp}
Let $m\geq 1$ be an integer. When $t\downarrow 0$, we have the following
asymptotic expansion for $\theta _{U}\left( t\right) $
\begin{multline}  \label{eq: theta asympt}
\theta _{U}\left( t\right) =\frac{\log D+\log a-2\gamma }{%
2\log a}-\frac{\log t}{\log a}+\sum\limits_{k=0}^{M}c_{a,b}\left(
k,t\right) t^{k\left( 1-\frac{\log |b|}{\log a}\right) -\frac{l_{Q,k}\pi i}{\log a%
}}+\\
+\sum\limits_{\ell\in B(m,a,b)}(d_{a,b}\left( \ell\right)-\tilde{d}_{a,b}\left(\ell \right)\log t ) t^{\ell}+ \sum_{\ell\in\{1,\ldots,2m\}\setminus B(m,a,b) }\frac{(-1)^\ell}{\ell!}\zeta_U(-\ell)t^\ell  +O\left( t^{2m+c_{0}}\right),
\end{multline}
where $M=\lfloor2m/\left(1-\frac{\log|b|}{\log a}\right) \rfloor$, $c_0 >0$ is an absolute constant which depends only upon the zeta function,
\begin{equation} \label{eq: defn of c a,b}
c_{a,b}\left( k,t\right) =\frac{D^{\frac{k}{2}\left( \frac{\log |b|}{\log a}%
-1\right) +\frac{l_{Q,k}\pi i}{2\log a}}}{k!\log a}\sum\limits_{n\in A_{a,b}(k)}D^{\frac{n\pi i}{\log a}}\Gamma \left( \frac{\log |b|}{\log a}k+\frac{%
\left(2n+l_{Q,k}\right)\pi i}{\log a}\right) t^{-\frac{2n\pi i}{\log a}},
\end{equation}%
and $d_{a,b}\left( \ell\right) $, $\tilde{d}_{a,b}\left(\ell\right)$, for $\ell= k\left(1-\frac{\log|b|}{\log a}\right)\in B(m,a,b)$ are given by
\begin{multline} \label{eq: defn of d a,b}
d_{a,b}\left( \ell\right) =\frac{(-1)^\ell}{\ell!D^{\ell/2}}\left[\frac{(-1)^k}{2\log a}\binom{\ell}{k}\left(\log D-\log a +2(H_{\ell-k}-\gamma)\right) + \right. \\ +\left. \sum_{j\in\{0,\ldots,\ell\}\setminus\{k\}}\binom{\ell}{j}\frac{(-1)^ja^{\ell-j}b^j}{1-a^{\ell-j}b^j} \right]
\end{multline}
and
\begin{equation} \label{eq: defn of d a,b tilde}
\tilde{d}_{a,b}(\ell)=\frac{(-1)^{\ell+k}}{\ell!}D^{-\ell/2}\frac{1}{\log a}\binom{\ell}{k}.
\end{equation}
When $B(m,a,b)=\emptyset$, the sum over $\ell\in B(m,a,b)$ on the right hand side of \eqref{eq: theta asympt} is identically zero.

Moreover, the series \eqref{eq: defn of c a,b} is absolutely convergent and uniformly bounded by a constant independent of $t$, as $t\downarrow 0$.
\end{thm}

\begin{proof}
The function $\zeta _{U}\left( s\right) $ is holomorphic and absolutely convergent
for $\Re(s) >0$, hence the Mellin inversion formula applied to $\zeta
_{U}\left( s\right) \Gamma \left( s\right)$ yields the representation
\begin{equation*}
\theta _{U}\left( t\right)-1 =\frac{1}{2 \pi i}\int\limits_{c-i\infty }^{c+i\infty }\zeta
_{U}\left( s\right) \Gamma \left( s\right) t^{-s}ds = \frac{1}{2 \pi i} \underset{T\rightarrow \infty }{\lim }%
 \int\limits_{c-iT}^{c+iT}\zeta _{U}\left( s\right) \Gamma \left( s\right)
t^{-s}ds,
\end{equation*}%
valid for some $c\in \mathbb{R}_{> 0}$.

The gamma function $\Gamma \left( s\right) $ has simple poles at the negative integers and zero. The Laurent series expansion of $\Gamma \left( s\right) $ at $s=-\ell$ ($\ell\in \mathbb{N}_{0}$) is given by
\begin{equation}\label{eq: gamma exp}
\Gamma(s)=\frac{\left( -1\right) ^{\ell}}{\ell!}\left( \frac{1}{s+\ell} + H_\ell-\gamma \right)+ O(s+\ell), \text{  as  } s+\ell\to 0.
\end{equation}
Therefore, the function $\zeta _{U}\left( s\right) \Gamma \left( s\right)
t^{-s}$ has a pole of order $2$ at $s=0$, poles of order at most 2 at negative integers and at most simple poles at $s=s_{k,n}$  defined by \eqref{eq: s k,n defn}, when $\frac{k\log
\left\vert Q\right\vert }{\log a}\notin \mathbb{\mathbb{Z}}$.

Let the constant $c_{0}\in\left( 0,1\right) $ be chosen so that $-2m-c_{0}$ does not coincide with poles of
$\zeta _{U}\left( s\right) \Gamma \left( s\right) t^{-s}$ for every $m\in
\mathbb{N}$. Let $T>0$ be such that the function $\zeta _{U}\left( s\right)
\Gamma \left( s\right) t^{-s}$ is analytic on the boundary of the rectangle $R(c,m,c_0,T)$ with vertices $%
c\pm iT$ and $-2m-c_{0}\pm iT$ and meromorphic inside it. The Residue Theorem gives us
\begin{multline} \label{Residue Th}
\int\limits_{c-iT}^{c+iT}\zeta _{U}\left( s\right) \Gamma
\left( s\right) t^{-s}ds = 2\pi i\sum_{s\in R(c,m,c_0,T)} \mathrm{Res}\left( \zeta _{U}\left(
s\right) \Gamma \left( s\right) t^{-s}\right)
\\-\int\limits_{c+iT}^{-2m-c_{0}+iT}\zeta _{U}\left( s\right) \Gamma \left(
s\right) t^{-s}ds -\int\limits_{-2m-c_{0}-iT}^{c-iT}\zeta _{U}\left(
s\right) \Gamma \left( s\right)
t^{-s}ds+\int\limits_{-2m-c_{0}-iT}^{-2m-c_{0}+iT}\zeta _{U}\left( s\right)
\Gamma \left( s\right) t^{-s}ds\text{ .}
\end{multline}
For large real numbers $x,y$, from \cite[formula 8.328 on p. 904]{knjiga integrala} we have the estimate
\begin{equation} \label{eq. gamma estimate}
\left\vert \Gamma \left( x+iy\right) \right\vert \sim \sqrt{2\pi }e^{-\frac{\pi }{2}\left\vert y\right\vert }\left\vert y\right\vert ^{-\frac{1}{2}+x}.
\end{equation}
This, combined with the fact that $|b|/a <1$ yields that for any $\tau\in\mathbb{R}$ one has
\begin{align*}
 |\zeta_U(-2m-c_0+i\tau)|&\leq  D^{-\frac{2m+c_0}{2}}\sum_{k=0}^\infty \frac{|\Gamma(2m+c_0+1-i\tau)|}{k! |\Gamma(2m+c_0+1-k-i\tau)|}\frac{a^{2m+c_0} (|b|/a)^k}{\left|1- a^{2m+c_0-i\tau}(b/a)^k\right|} \\
   &\ll \left(\frac{a}{\sqrt{D}}\right)^{2m+c_0}\sum_{k=0}^\infty \frac{(|\tau| |b|/a)^k}{k!}\ll \left(\frac{a}{\sqrt{D}}\right)^{2m+c_0}e^{|\tau|}.
\end{align*}
Now, we easily deduce that the two integrals on the right-hand side of \eqref{Residue Th} taken over the horizontal lines tend to zero when $T\rightarrow \infty $ and that the integral along the vertical line, when $t\downarrow 0$ can be estimated as
\begin{eqnarray*}
\left\vert \int\limits_{-2m-c_{0}-iT}^{-2m-c_{0}+iT}\zeta _{U}\left(
s\right) \Gamma \left( s\right) t^{-s}ds\right\vert \ll t^{2m+c_{0}} \int\limits_{1}^{\infty}e^{(1-\pi/2)\tau} d\tau=O(t^{2m+c_0}),
\end{eqnarray*}
where the bound is uniform in $T$. Therefore, letting $T\to \infty$ in \eqref{Residue Th} yields that
\begin{equation}\label{eq: theta precomputation}
\theta_U(t)=1+\lim_{T\to \infty}\sum_{s\in R(c,m,c_0,T)} \mathrm{Res}\left( \zeta _{U}\left(
s\right) \Gamma \left( s\right) t^{-s}\right) + O(t^{2m+c_0}), \text{ as } t\downarrow 0,
\end{equation}
under assumption that the limit of the sum over residues is finite. This limit equals the sum over residues of the function $ \zeta _{U}\left(
s\right) \Gamma \left( s\right) t^{-s}$ in the strip $-2m-c_0\leq \Re(s) \leq c$. Now, we will compute the sum of residues of this function at its poles in the strip $-2m-c_0\leq \Re(s) \leq c$ and prove it is finite.

For $k\geq 0$ and $n\in\mathbb{Z}$, the pole $s=s_{k,n}=-k\left( 1-\frac{\log |b|}{\log a}\right)+\frac{\left( 2n+l_{Q,k}\right) \pi i}{\log a}$ of $\zeta_U(s)$ belongs to the strip $-2m-c_0\leq \Re(s) \leq c$ if and only if $k\in\{0,1,\ldots,M\}$, where $M=\lfloor2m/\left(1-\frac{\log|b|}{\log a}\right) \rfloor$. Therefore, all poles $w$ of $ \zeta _{U}\left(s\right) \Gamma \left( s\right) t^{-s}$ in the strip $-2m-c_0\leq \Re(s) \leq c$ are of the form $w=s_{n,k}$, for $k\in\{0,1,\ldots,M\}$ and $n\in \mathbb{Z}$ and of the form $w=-j$, $j\in\{0,\ldots,2m\}$, where poles $w=s_{n,k}\notin \mathbb{Z}_{<0}$ and $w=-j\notin B(m,a,b)$ are simple, while poles $w=0$ and $w=-\ell\in B(m,a,b)$ are double poles. Let us compute residues at all those poles. First, we treat simple poles (there are infinitely many such poles in our strip).

If $k\in\{0,1,\ldots,M\}$ and $n\in\mathbb{Z}$ are such that $s_{k,n}\notin \{-2m,-2m+1,\ldots,-1\}$, then the function $\zeta _{U}\left( s\right)\Gamma \left( s\right) t^{-s}$ has a simple pole at $s=s_{k,n}$ with the residue
\begin{gather*}
\underset{s=s_{k,n}}{\mathrm{Res}}\left( \zeta _{U}\left( s\right) \Gamma
\left( s\right) t^{-s}\right) =D^{-k+k\frac{\log \left\vert Q\right\vert }{%
2\log a}+\frac{\left( 2n+l_{Q,k}\right) \pi i}{\log a}}\binom{-2k+\frac{%
k\log \left\vert Q\right\vert }{\log a}+\frac{\left( 2n+l_{Q,k}\right) \pi i%
}{\log a}}{k} \\
\cdot \frac{\left( -1\right) ^{k}}{\log a}\Gamma \left( -2k+\frac{k\log
\left\vert Q\right\vert }{\log a}+\frac{\left( 2n+l_{Q,k}\right) \pi i}{\log
a}\right) t^{2k-\frac{k\log \left\vert Q\right\vert }{\log a}-\frac{\left(
2n+l_{Q,k}\right) \pi i}{\log a}} \\
=D^{\frac{k}{2}\left( \frac{\log |b|}{\log a}-1\right) +\frac{\left(2n+l_{Q,k}\right)\pi i}{%
2\log a}}\frac{1}{k!\log a}\Gamma \left( \frac{\log |b|}{\log a}k+\frac{%
\left(2n+l_{Q,k} \right)\pi i}{\log a}\right) t^{k\left( 1-\frac{\log |b|}{\log a}\right) -%
\frac{\left(2n+l_{Q,k} \right)\pi i}{\log a}}.
\end{gather*}
The asymptotic relation \eqref{eq. gamma estimate} yields the following estimate for any integer $k\geq 0$:
\begin{multline*}
\sum\limits_{n\in
A_{a,b}(k)}\left\vert \underset{s=s_{k,n}}{\mathrm{Res}}\left( \zeta _{U}\left(
s\right) \Gamma \left( s\right) t^{-s}\right) \right\vert \ll \frac{%
\sqrt{2\pi }}{k!}\left( \log a\right) ^{-\frac{1}{2}-\frac{\log |b|}{\log a}%
k}D^{\frac{k}{2}\left( \frac{\log |b|}{\log a}-1\right) }t^{k\left( 1-\frac{%
\log |b|}{\log a}\right) } \\
\cdot \sum\limits_{n\in A_{a,b}(k)\setminus\{-\tfrac{l_{Q,k}}{2}\}}e^{-\frac{\pi ^{2}}{2\log a}\left\vert 2n+l_{Q,k}\right\vert }\left\vert
2n+l_{Q,k}\right\vert ^{-\frac{1}{2}+\frac{\log |b|}{\log a}k},
\end{multline*}
which proves that the series $\sum\limits_{n\in A_{a,b}(k)}\underset{s=s_{k,n}}{\mathrm{Res}}\left( \zeta _{U}\left( s\right) \Gamma
\left( s\right) t^{-s}\right) $ is absolutely convergent.

For a fixed $k\in\{0,1,\ldots,M\}$ and all integers $n$ such that $s_{n,k}\notin\mathbb{Z}_{<0}$ we can write
\begin{equation*}
\sum\limits_{n\in A_{a,b}(k)}\underset{s=s_{k,n}}{\mathrm{Res}}\left( \zeta _{U}\left( s\right) \Gamma
\left( s\right) t^{-s}\right) =c_{a,b}\left( k,t\right) t^{k\left( 1-\frac{%
\log |b|}{\log a}\right) -\frac{l_{Q,k}\pi i}{\log a}},
\end{equation*}%
where $c_{a,b}\left( k,t\right)$ is defined by \eqref{eq: defn of c a,b}.

If $-s=\ell\notin B(m,a,b)$ for an integer $\ell\in\{1,\ldots,2m\}$, then $-\ell$ is a simple pole of $\zeta_U(s)\Gamma(s)t^{-s}$ with the residue $\frac{(-1)^\ell}{\ell!}\zeta_U(-\ell)t^\ell$, hence the contribution of all such poles to the sum of all residues in the strip $-2m-c_0\leq \Re (s) \leq c$ is given by
\begin{equation} \label{eq: simple poles at l}
\sum_{\ell\in\{1,\ldots,2m\}\setminus B(m,a,b) }\frac{(-1)^\ell}{\ell!}\zeta_U(-\ell)t^\ell.
\end{equation}

It is left to evaluate contribution from double poles. Let us start with the double pole at $0$. Multiplying the Laurent series expansions \eqref{eq: series for zetaU at s=0} and \eqref{eq: gamma exp} (for $\ell=0$) with the Taylor series expansion $t^{-s}=e^{-s\log t}=1-s\log t+O\left( s^{2}\right) $, we get
\begin{equation*}
\underset{s=0}{\mathrm{Res}}\left( \zeta _{U}\left( s\right) \Gamma \left(
s\right) t^{-s}\right) =\frac{\log D-\log a-2\gamma }{2\log a}-\frac{\log t}{%
\log a}.
\end{equation*}

Now, we consider the case when $B(m,a,b)\neq \emptyset$, i.e. when  $s=s_{k,n}=-\ell$, for $\ell\in \mathbb{N}$. In this case, $-\ell$ is a double pole of the function $\zeta _{U}\left( s\right) \Gamma \left(s\right) t^{-s}$.  In this case, we write $\zeta_U(s)$ for $s$ close to $-\ell$ as
\begin{multline}\label{eq: Laurent series zeta at -l}
\zeta_U(s)= D^{-\ell/2}\binom{\ell}{k}\frac{(-1)^k}{\log a}\left(\frac{1}{s+\ell} + \left( \frac{\log D - \log a}{2} + H_{\ell-k} - H_{\ell}\right) \right)+\\ + D^{-\ell/2}\sum_{j=0, j\neq k}^{\ell}\binom{\ell}{j}(-1)^j\frac{a^{\ell-j} b^j}{1-a^{\ell-j} b^j}+O\left(s+\ell\right),
\end{multline}
and use the Taylor series expansion $t^{-s}= t^{\ell}(1-(s+\ell)\log t + O((s+\ell)^2))$ for $s$ close to $-\ell$, and the Laurent series expansion \eqref{eq: gamma exp} to deduce that
\begin{multline*}
\underset{s=-\ell}{\mathrm{Res}}\left(\zeta _{U}\left( s\right)\Gamma(s)t^{-s}\right) =\frac{(-1)^{\ell}}{\ell!}D^{-\frac{\ell}{2}}\frac{\left( -1\right) ^{k}}{\log a}\binom{\ell}{k}t^{\ell}\left( \frac{\log D - \log a}{2} + H_{\ell-k} - \gamma - \log t\right)\\
+ \frac{(-1)^{\ell}}{\ell!}D^{-\frac{\ell}{2}} t^{\ell}\sum_{j=0, j\neq k}^{\ell}\binom{\ell}{j}(-1)^j\frac{a^{\ell-j} b^j}{1-a^{\ell-j} b^j}= (d_{a,b}\left( \ell\right)-\tilde{d}_{a,b}\left( \ell\right)\log t ) t^{\ell},
\end{multline*}
where $d_{a,b}(\ell)$ and $\tilde{d}_{a,b}(\ell)$ are given by \eqref{eq: defn of d a,b} and \eqref{eq: defn of d a,b tilde}, respectively.

Now, we see that the sum of all residues of the function $\zeta _{U}\left( s\right)
\Gamma \left( s\right) t^{-s}$ inside the strip $-2m-c_0\leq \Re(s) \leq c$ equals
\begin{multline*}
\sum_{\underset{ -2m-c_0\leq \Re(s) \leq c }{\text{all poles } s,}}\mathrm{Res}\left( \zeta _{U}\left( s\right)
\Gamma \left( s\right) t^{-s}\right) =\frac{\log D-\log a-2\gamma }{2\log a} +\sum\limits_{k=0}^{M}c_{a,b}\left(
k,t\right) t^{k\left( 1-\frac{\log |b|}{\log a}\right) -\frac{l_{Q,k}\pi i}{\log a%
}}\\ +\sum\limits_{\ell\in B(m,a,b)}(d_{a,b}\left( \ell\right)-\tilde{d}_{a,b}\left( \ell\right)\log t )t^\ell -\frac{\log t}{\log a} +\sum_{\ell\in\{1,\ldots,2m\}\setminus B(m,a,b) }\frac{(-1)^\ell}{\ell!}\zeta_U(-\ell)t^{\ell},
\end{multline*}
where $c_{a,b}\left(k,t\right)$ is defined by \eqref{eq: defn of c a,b} and $d_{a,b}\left( \ell\right)$ and $\tilde{d}_{a,b}\left( \ell\right)$ are defined by \eqref{eq: defn of d a,b} and \eqref{eq: defn of d a,b tilde}, respectively. Combining this with \eqref{eq: theta precomputation} completes the proof.
\end{proof}

\section{The Hurwitz-type zeta function associated to the Lucas sequence}

In this section we will prove that the Hurwitz-type zeta function associated to the Lucas sequence $\{U_{n}\}_{n\geq 0}$,  defined for $\Re (s)>0$ and $\Re (z)>0$ by \eqref{eq: zeta defn}, for all $z\in\mathbb{C}\setminus (-\infty,0]$ can be meromorphically continued to the whole complex $s-$plane and we will identify its polar structure.

The starting point of our investigation is the fact that $\zeta_U(s,z)$, for $\Re (s)>0$ and $\Re (z)>0$, can be expressed as the Laplace-Mellin transform of the theta function $\theta_U(t)$:
\begin{equation}
\zeta _{U}\left( s,z\right) =
\frac{1}{\Gamma \left( s\right) }\int_{0}^{\infty }\theta
_{U}\left( t\right) e^{-zt}t^{s}\frac{dt}{t}.  \label{Hurwitz}
\end{equation}

We define the subset $B(a,b)$ of positive integers by
$$B(a,b)=\bigcup_{m \in \mathbb{N}}B(m,a,b).$$

The main result of the paper is the following theorem.
\begin{thm}\label{thm: main}
The Hurwitz-type zeta function $\zeta _{U}\left( s,z\right) $, for $z\in \mathbb{C}\setminus
\left( -\infty ,0\right] $, can be continued to a meromorphic function on the complex $s-$plane, with simple poles at $s=-\ell $
and $s_{\ell,k,n}=-\ell -k\left( 1-\frac{\log |b|}{\log a}\right) + \frac{%
\left(2n+l_{Q,k} \right)\pi i}{\log a}$, for $\ell, k \in \mathbb{N}_{0}, \ n\in A_{a,b}(k)$, with the corresponding residues given by $$\frac{z^{\ell}}{\log a}+\sum\limits_{j \in B(a,b), \ j\leq \ell}\tilde{d}_{a,b}\left( j\right)\left(-\ell\right)_jz^{\ell-j}$$
and
\begin{equation}\label{eq: residue Hzeta big}
C_{a,b}\left(k\right) D^{\frac{n\pi i }{\log a}}\frac{\left( -z\right)
^{\ell}}{\ell!}\prod%
\limits_{j=1}^{\ell+k}\left( k\frac{\log |b|}{\log a} +\frac{\left(2n+l_{Q,k}\right) \pi i}{\log a}-j\right),
\end{equation}
respectively, where
\begin{equation} \label{eq: defn Ca,b}
C_{a,b}(k)=\frac{D^{\frac{k}{2}\left( \frac{\log |b|}{\log a}-1\right) +\frac{%
l_{Q,k}\pi i}{2\log a}}}{k!\log a},
\end{equation}
with the convention that if $s_{\ell_{1},k_{1},n_{1}}=s_{\ell_{2},k_{2},n_{2}}$ for different triples $(\ell_{m},k_m,n_m)$, $m=1,2$, the corresponding residues are added.
\end{thm}

\begin{proof}
We start with continuation of $\zeta _{U}\left( s,z\right)$ in $z-$variable, for a fixed $s$ with $\Re (s)>0$. For an arbitrary $m \in \mathbb{N}$, it suffices to prove meromorphic continuation to the half plane $\Re (z)>-m$ with the cut along the negative real axis.

By Lemma \ref{lem2} (i) there exist $N\in\mathbb{N}$ and $K>0$ such that

\begin{equation}  \label{eq: AS3 cond}
\left\vert \sum\limits_{n=N}^{\infty }e^{-U_{n}t}\right\vert \leq
Ke^{-\left(m+1\right) t}.
\end{equation}

For such $N$, the Hurwitz-type zeta function can be written as
\begin{eqnarray*}
\zeta _{U}\left( s,z\right) &=&\frac{1}{\Gamma \left( s\right) }\left[
\int_{0}^{\infty }e^{-zt}t^{s}\frac{dt}{t}+\sum\limits_{n=1}^{N-1}\int_{0}^{%
\infty }e^{-U_{n}t}e^{-zt}t^{s}\frac{dt}{t}+\int_{0}^{\infty
}\sum\limits_{n=N}^{\infty }e^{-U_{n}t}e^{-zt}t^{s}\frac{dt}{t}\right] \\
&=&\frac{1}{z^{s}}+\sum\limits_{n=1}^{N-1}\frac{1}{\left( z+U_{n}\right) ^{s}%
}+\frac{1}{\Gamma \left( s\right) }\int_{0}^{\infty
}\sum\limits_{n=N}^{\infty }e^{-U_{n}t}e^{-zt}t^{s}\frac{dt}{t}.
\end{eqnarray*}
The first two summands are meromorphic functions of $z\in \mathbb{C}%
\setminus \left( -\infty ,0\right] $. Let us consider the third summand. The
bound \eqref{eq: AS3 cond} yields that the integral converges uniformly in $%
z $ on every compact subset of the half-plane $\Re (z)>-m$, and hence
represents a holomorphic function in $z$ in this half-plane. This completes
the proof of meromorphic continuation of $\zeta _{U}\left( s,z\right)$ in $%
z- $variable, for a fixed $s$ with $\Re (s)>0$.

Now we turn our attention to meromorphic continuation of Hurwitz-type zeta function $\zeta _{U}\left( s,z\right) $ in $s-$variable. We start with the
representation
\begin{equation}  \label{pos}
\begin{aligned} \zeta _{U}\left( s,z\right) &=
\frac{1}{\Gamma(s)}\int_{0}^{1}\left( \theta _{U}\left( t\right) -\alpha
_{a,b,m}\left( t\right) \right) e^{-zt}t^{s}\frac{dt}{t}\\
&+\frac{1}{\Gamma(s)}\int_{1}^{\infty }\left( \theta _{U}\left( t\right)
-\alpha _{a,b,m}\left( t\right) \right)
e^{-zt}t^{s}\frac{dt}{t}+\frac{1}{\Gamma(s)}\int_{0}^{\infty }\alpha
_{a,b,m}\left( t\right) e^{-zt}t^{s}\frac{dt}{t}, \end{aligned}
\end{equation}
where we put
\begin{equation*}
\begin{split}
\alpha _{a,b,m}\left( t\right) &=\frac{\log D+\log a-2\gamma }{%
2\log a}-\frac{\log t}{\log a}+\sum\limits_{k=0}^{M}c_{a,b}\left(
k,t\right) t^{k\left( 1-\frac{\log |b|}{\log a}\right) -\frac{l_{Q,k}}{\log a%
}\pi i}+\\
&+\sum\limits_{\ell=1}^{2m}e_{a,b}\left( \ell\right)t^{\ell}-\sum\limits_{\ell \in B(m,a,b)}\tilde{d}_{a,b}\left(\ell \right)t^{\ell}\log t,
\end{split}
\end{equation*}
where
$$
e_{a,b}(\ell)=\left\{
           \begin{array}{ll}
             d_{a,b}(\ell), & \text{ if  } \ell \in B(m,a,b); \\
             \frac{(-1)^\ell}{\ell!}\zeta_U(-\ell), & \text{ if  } \ell\in\{1,\ldots,2m\}\setminus B(m,a,b).
           \end{array}
         \right.
$$

Equation \eqref{eq: theta asympt} yields the representation
\begin{equation*}
\theta _{U}\left( t\right) =\alpha _{a,b,m}\left( t\right) +O\left(
t^{2m+c_{0}}\right) ,
\end{equation*}
hence, the first integral in (\ref{pos}) is holomorphic for $\Re (s)>-\left(
2m+c_{0}\right) $.

The second integral on the right-hand side of \eqref{pos} is holomorphic for $%
\Re (z)>0$ and all $s\in \mathbb{C}$. Its continuation to the cut $%
z-$plane is done in the same way as above, hence, we may consider it to be
holomorphic in the whole $s-$plane, for $z\in \mathbb{C}\setminus \left(
-\infty ,0\right]$.

The third integral on the right-hand side of \eqref{pos} can be written as
\begin{equation*}
\frac{1}{\Gamma (s)}\int_{0}^{+\infty }\alpha _{a,b,m}\left( t\right)
e^{-zt}t^{s}\frac{dt}{t}:=I_{1}-I_{2}+I_{3}+I_{4}-I_{5},
\end{equation*}%
where
\begin{eqnarray*}
I_{1} &=&\frac{\log D+\log a-2\gamma }{2 \log a\Gamma (s)}%
\int_{0}^{+\infty }e^{-zt}t^{s-1}dt=\frac{\log D+\log a-2\gamma }{2 \log a}%
\cdot \frac{1}{z^{s}}; \\
I_{2} &=&\frac{1}{\log a\Gamma (s)}\int_{0}^{+\infty
}e^{-zt}t^{s-1}\log tdt=\frac{1}{\log a}\frac{\psi
\left( s\right) -\log z}{z^{s}} ; \\
I_{3} &=&\frac{1}{\Gamma (s)}\sum\limits_{k=0}^{M}\int_{0}^{+\infty
}c_{a,b}\left( k,t\right) t^{k\left( 1-\frac{\log |b|}{\log a}\right) -\frac{%
l_{Q,k}}{\log a}\pi i}e^{-zt}t^{s-1}dt; \\
I_{4} &=&\frac{1}{\Gamma (s)}\sum\limits_{\ell=1}^{2m}e_{a,b}\left(
\ell\right) \int_{0}^{+\infty }e^{-zt}t^{s+\ell-1}dt=\sum\limits_{\ell=1}^{2m}e_{a,b}\left( \ell\right) \frac{\left( s \right)_\ell
}{ z^{s+\ell}}; \\
I_{5} &=&\frac{1}{\Gamma (s)}\sum\limits_{\ell \in B(m,a,b)}\tilde{d}_{a,b}\left(
\ell\right) \int_{0}^{+\infty }e^{-zt}t^{s+\ell-1}\log t dt \\ &=&\sum\limits_{\ell \in B(m,a,b)}\tilde{d}_{a,b}\left( \ell\right) \frac{\left( s \right)_\ell
 \left(\psi\left( s+\ell\right) -\log z\right)}{ z^{s+\ell}}.
\end{eqnarray*}%
Integrals $I_1$, $I_2$, $I_4$ and $I_5$ were deduced using \cite[formulae 3.381.4 on p. 346 and
4.352.1 on p. 573]{knjiga integrala}, where $\left( s \right)_\ell$ is the Pochhammer symbol.

Integrals $I_1$ and $I_4$ are obviously holomorphic functions of $s$. It is known that digamma function $\psi \left( s\right) $ is meromorphic with simple poles at $s=-\ell$, $\ell\in
\mathbb{N}_{0}$ and corresponding residues $-1$ (e.g. see \cite[p. 24]{zeta}). Therefore, integral $I_2$ is meromorphic in $s$ with simple poles at zero and negative integers, and integral $I_5$ is meromorphic in $s$, with simple poles at negative integers $j$ such that $j\leq -\ell$, for $\ell\in B(m, a,b)$.

It is left to evaluate $I_{3}$ above, show it is a meromorphic function in $s$ and deduce its poles.
Recall that
\begin{equation*}
c_{a,b}\left( k,t\right) =C_{a,b}(k)\sum\limits_{n\in A_{a,b}(k)}D^{\frac{n\pi i }{\log a}}\Gamma \left( \frac{\log |b|}{\log a}k+\frac{%
(2n+l_{Q,k})\pi i}{\log a}\right) t^{-\frac{2n\pi i}{\log a}},
\end{equation*}%
where $C_{a,b}(k)$ is defined by \eqref{eq: defn Ca,b}.

Therefore,
\begin{multline*}
\int_{0}^{+\infty }c_{a,b}\left( k,t\right) t^{k\left( 1-\frac{\log |b|}{%
\log a}\right) -\frac{l_{Q,k}\pi i}{\log a}}e^{-zt}t^{s-1}dt= C_{a,b}(k) \cdot\\
\cdot\sum\limits_{n\in A_{a,b}(k)}D^{\frac{n\pi i}{\log a}}\Gamma
\left( \frac{\log |b|}{\log a}k+\frac{(2n+l_{Q,k})\pi i}{\log a}\right)
\int_{0}^{+\infty }t^{k\left( 1-\frac{\log |b|}{\log a}\right) -\frac{%
(2n+l_{Q,k})\pi i}{\log a} +s-1}e^{-zt}dt,
\end{multline*}
where interchanging the sum and the integral in the equation above is
justified for $\Re (z)>0$ by the bound

\begin{multline*}
\sum\limits_{n\in A_{a,b}(k)}e^{-\frac{\pi ^{2}}{2\log a}\left\vert
2n+l_{Q,k}\right\vert }\left\vert 2n+l_{Q,k}\right\vert ^{-\frac{1}{2}+\frac{%
\log |b|}{\log a}k}\int_{0}^{+\infty }t^{k\left( 1-\frac{\log |b|}{\log a}%
\right) +\Re (s)-1}e^{-\Re (z)t}dt \\
\leq \sum\limits_{n\in A_{a,b}(k)}e^{-\frac{\pi ^{2}}{2\log a}\left\vert
2n+l_{Q,k}\right\vert }\left\vert 2n+l_{Q,k}\right\vert ^{-\frac{1}{2}+\frac{%
\log |b|}{\log a}k}\frac{\Gamma \left( k\left( 1-\frac{\log |b|}{\log a}%
\right) +\Re (s)\right) }{(\Re (z))^{k\left( 1-\frac{\log |b|}{\log a}\right)
+\Re (s)}},
\end{multline*}%
and the fact that each term on the right-hand side of the above display
decays exponentially. This proves that
\begin{multline*}
I_{3}=\frac{1}{\Gamma (s)}\sum\limits_{k=0}^{M}C_{a,b}(k)\sum\limits_{n \in A_{a,b}(k)}D^{\frac{n\pi i}{\log a}}\Gamma \left( \frac{\log |b|}{\log a}%
k+\frac{(2n+l_{Q,k})\pi i}{\log a}\right) \\
\cdot \frac{\Gamma \left( k\left( 1-\frac{\log |b|}{\log a}\right) -\frac{%
(2n+l_{Q,k})\pi i}{\log a}+s\right) }{z^{k\left( 1-\frac{\log |b|}{\log a}%
\right) -\frac{(2n+l_{Q,k})\pi i}{\log a}+s}}.
\end{multline*}
Therefore, $I_3$ is meromorphic, with simple poles at
\begin{equation} \label{eq: s lnk}
s_{\ell, k, n}=-\ell -k\left( 1-\frac{\log |b|}{\log a}\right) + \frac{(2n+l_{Q,k}) \pi i}{%
\log a},\ \ell, k \in \mathbb{N}_{0}, \ n\in A_{a,b}(k).
\end{equation}
This proves meromorphic continuation of the third term on the right-hand side of \eqref{pos} to all $z$ in the cut plane $\mathbb{C}\setminus (-\infty,0]$ and all $s$ with $\Re(s)> -m$.

Since $m \in \mathbb{N}$ was arbitrarily chosen, the above analysis yields that $\zeta_U(s,z)$ is meromorphic in $s\in
\mathbb{C}$, and holomorphic in $z\in \mathbb{C}\setminus (-\infty ,0]$,
with simple poles at non-positive integers and at $s_{\ell, k, n}$ given by \eqref{eq: s lnk}. A simple computation shows that
\begin{eqnarray*}
\underset{%
\begin{array}{c}
s=-\ell
\end{array}%
}{\mathrm{Res}}\zeta _{U}\left( s,z\right) &=&\frac{z^{\ell}}{\log a}+\sum\limits_{j \in B(a,b), \ j \leq \ell}\tilde{d}_{a,b}\left( j\right)\left(-\ell\right)_jz^{\ell-j}, \ \ell\in \mathbb{N}_{0}, \\
\underset{%
\begin{array}{c}
s=s_{\ell,k,n}
\end{array}%
}{\mathrm{Res}}\zeta _{U}\left( s,z\right) &=&C_{a,b}\left(k\right) D^{\frac{n\pi i }{\log a}}\frac{\left( -z\right)
^{\ell}}{\ell!}\prod%
\limits_{j=1}^{\ell+k}\left( k\frac{\log |b|}{\log a} +\frac{(2n+l_{Q,k})\pi i}{%
\log a}-j\right), \\
&&\ell, k \in \mathbb{N}_{0}, \ n\in A_{a,b}(k).
\end{eqnarray*}

The proof is complete.
\end{proof}

\section{Examples and concluding remarks}

In this section we present two interesting examples and give remarks on zeta functions associated to order one and order three recurrence sequences. In the first example we consider the Hurwitz-type zeta function associated to the Lucas sequence $\{U_n\}_{n\geq 0}$, where $U_n$, for $n\geq 1$ is the sum of the first $n$ terms in the divergent geometric series $\sum_{k\geq 0} a^k$, $a>1$. In the second example we take $U_n$ to be the Fibonacci sequence.

Then, we will discuss the Hurwitz-type zeta function associated to the divergent geometric sequence, which arises in the situation when $Q=0$ (and $a>1$). We will end the paper with a remark on a follow-up of this paper, related to zeta functions associated to recurrence sequences of order three.

\subsection{Examples}

\begin{example}\rm
Let $b=1$, $a>1$ and set $U_0=1$; $U_n=\sum_{j=0}^{n-1}a^j$, for $n\geq 1$. Then, by Theorem \ref{thm: main} the Hurwitz-type zeta function
\begin{equation} \label{eq: exmaple zeta}
\zeta_U(s,z)=\frac{1}{(z+1)^s}+\sum_{n=1}^{\infty}\frac{1}{(z+(1+\ldots + a^{n-1}))^s},
\end{equation}
initially defined for $\Re(s)>0$ possesses, for all $z\in\mathbb{C}\setminus (-\infty,0]$ a meromorphic continuation to the complex $s-$plane, with simple poles at all non-positive integers $s=-\ell$, $\ell\in\mathbb{N}_0$, and at numbers $s_{m,n}= -m + \frac{2\pi i n}{\log a}$, $m\in\mathbb{N}_0$, $n\in\mathbb{Z}\setminus \{0\}$.

First, we compute residues at $s=-\ell$, $\ell\in\mathbb{N}_0$. Since $b=1$, the set $B(a,b)=B(a,1)$ equals $\mathbb{N}$; moreover, $\tilde{d}_{a,1}(j)=\frac{1}{j! \log a}D^{-j/2}=\frac{1}{j! \log a} (a-1)^{-j}$. Now, we may conclude that
\begin{align*}
\underset{s=-\ell}{\mathrm{Res}}\zeta _{U}\left( s,z\right)& =\frac{1}{\log a}\left( z^{\ell} + \sum_{j=1}^{\ell} \frac{(-\ell)_j }{j!} (a-1)^{-j} z^{\ell-j}\right) \\&= \frac{1}{\log a} \sum_{j=0}^{\ell} \binom{\ell}{j}\left(\frac{1}{1-a}\right)^j  z^{\ell-j} = \frac{1}{\log a} \left(z+\frac{1}{1-a}\right)^\ell.
\end{align*}
We find it amusing that the "sum" of the divergent geometric series $\sum_{n=0}^{\infty}a^n$ appears in the residue, as if we took the limit as $n\to \infty$ in the expression for the $n$th term of the series defining $\zeta_U(s,z)$ at $s=-\ell$.

Now, we prove that the residue of the function \eqref{eq: exmaple zeta} at the pole $s_{m,n}= -m + \frac{2\pi i n}{\log a}$, for some fixed $m\in\mathbb{N}_0$, and $n\in\mathbb{Z}\setminus \{0\}$ is given by
\begin{equation}\label{eq: res at s m,n}
\underset{\begin{array}{c}
s=s_{m,n}
\end{array}}{\mathrm{Res}}\zeta _{U}\left( s,z\right)= \frac{(a-1)^{\tfrac{2n\pi i}{\log a}}}{\log a}\left(z+\frac{1}{1-a}\right)^m\frac{1}{m!}\prod_{j=1}^{m}\left(j-\frac{2n\pi i}{\log a}\right).
\end{equation}
We start with the observation that $s_{m,n}=s_{\ell,k,n}$, where $s_{\ell,k,n}$ is given by formula \eqref{eq: s lnk} if and only if $m=\ell + k$ and $n\in\mathbb{Z}\setminus \{0\}$. Therefore,
$$
\underset{\begin{array}{c}
s=s_{m,n}
\end{array}}{\mathrm{Res}}\zeta _{U}\left( s,z\right)= \sum_{\underset{\ell+k=m}{\ell,k\in\mathbb{N}_0}}
\underset{%
\begin{array}{c}
s=s_{\ell,k,n}
\end{array}%
}{\mathrm{Res}}\zeta _{U}\left( s,z\right).
$$
Since $D=(a-1)^2$, by letting $k=m-\ell$ in \eqref{eq: residue Hzeta big} and \eqref{eq: defn Ca,b} we get
$$
\underset{\begin{array}{c}
s=s_{m,n}
\end{array}}{\mathrm{Res}}\zeta _{U}\left( s,z\right)= \frac{(a-1)^{\tfrac{2n\pi i}{\log a}}}{\log a}\frac{1}{m!}\prod_{j=1}^{m}\left(\frac{2n\pi i}{\log a} -j \right)  \sum_{\ell=0}^{m}\frac{m!}{\ell! (m-\ell)!} (a-1)^{-(m-\ell)} (-z)^{\ell},
$$
hence
$$
\underset{\begin{array}{c}
s=s_{m,n}
\end{array}}{\mathrm{Res}}\zeta _{U}\left( s,z\right)= \frac{(a-1)^{\tfrac{2n\pi i}{\log a}}}{\log a}\frac{1}{m!}\prod_{j=1}^{m}\left(\frac{2n\pi i}{\log a} -j \right)  (-z-\frac{1}{1-a})^m,
$$
which proves \eqref{eq: res at s m,n}.
\end{example}

\begin{example}\rm
Let us investigate the case when $U_{n}=F_{n}$, where $\left\{
F_{n}\right\}_{n\geq 0} $ is the Fibonacci sequence. Then $a=\frac{\sqrt{5}+1}{
2}=\varphi$, $b=\frac{1-\sqrt{5}}{2} = \bar\varphi$. By Theorem \ref{thm: main}, the
Hurwitz-type zeta function
\begin{equation}
\zeta _{F}(s,z)=\sum_{n=0}^{\infty }\frac{1}{(z+F_{n})^{s}},
\label{eq: fib zeta}
\end{equation}%
initially defined for $\Re (s)>0$ possesses, for all $z\in \mathbb{C}
\setminus (-\infty ,0]$ a meromorphic continuation to the complex $s-$plane,
with simple poles at all non-positive integers $s=-\ell $, $\ell \in \mathbb{%
N}_{0}$, and at numbers $s_{\ell ,k,n}=-\ell -2k+\frac{\left( 2n+k\right)
\pi i}{\log a}$, $\ell $ $,k\in \mathbb{N}_{0}$, with $n\in\mathbb{Z}$ for odd $k$ and $n\in \mathbb{Z}%
\setminus \{-\frac{k}{2}\}$, for even $k$.

Let us compute the residues at $s=-\ell $, $\ell \in \mathbb{N}_{0}$.
Since $\frac{\log \left\vert b\right\vert }{\log a}=-1$, we have $
B(a,b)=\mathrm{4}\mathbb{N}$; moreover, $\tilde{d}_{a,b}(j)=\frac{\left(
-1\right) ^{\frac{j}{2}}}{j!\log a}5^{-\frac{j}{2}}\binom{j}{\frac{j}{2}}$.
Now, we may conclude that
\begin{equation}\label{res at minus l}
\underset{s=-\ell }{\mathrm{Res}}\zeta _{F}\left( s,z\right) =\frac{1}{\log a%
}\sum_{j=0}^{\left\lfloor \frac{\mathrm{\ell }}{4}\right\rfloor }\binom{4j}{%
2j}\binom{\ell }{4j}5^{-2j}z^{\ell -4j} =\frac{1}{\log a%
}\sum_{j=0}^{\left\lfloor \frac{\mathrm{\ell }}{4}\right\rfloor }\binom{\ell }{2j,2j}5^{-2j}z^{\ell -4j},
\end{equation}
where $\binom{\ell }{2j,2j}$ denotes the multinomial coefficient $\ell!/((2j)!(2j)! (\ell-2j)!)$.

We find it interesting to notice that the sum appearing on the right-hand side of \eqref{res at minus l} equals the sum of the terms in the trinomial $ (\sqrt{\varphi/5} + \sqrt{|\bar\varphi|/5}+ z)^{\ell}$ which possess rational coefficients.
\end{example}

\subsection{Concluding remarks}

\begin{remark}\rm
When $Q=0$, then the sequence associated to $(P,0)$ is of order one and reduces to the geometric sequence $U_n:= a^n$, $n\geq 0$. As above, we assume that $a>1$. The zeta function associated to the sequence $U=\{a^n\}_{n\geq 0}$ is the sum of the geometric series, i.e.
$$
\zeta_U(s)=\sum_{n=0}^{\infty} \frac{1}{a^{ns}}=\frac{1}{1-a^{-s}}, \text{  for  } \Re(s)>0.
$$
The right-hand side of the above equation obviously provides meromorphic continuation of $\zeta_U(s)$ to the whole complex plane with simple poles at numbers $s=\frac{2k\pi i}{\log a}$, $k\in\mathbb{Z}$.

The Hurwitz-type zeta function is defined by
\begin{equation}\label{eq: zeta Q=0}
\zeta_U(s,z)=\sum_{n=0}^{\infty} \frac{1}{(z+a^{n})^s}, \quad \Re(s)>0.
\end{equation}
Reasoning analogously as in Theorem \ref{thm: theta exp} above, we easily deduce the following expansion of the theta function $\theta_U(t):=\sum_{n\geq 0} \exp(-a^n t)$, as $t\downarrow 0$:
$$
\theta_U(t)=\frac{\log a -2\gamma}{2\log a} -\frac{\log t }{\log a} + \frac{2}{\log a}\Re\left( \sum_{k=1}^{\infty} \Gamma\left(\frac{2k\pi i}{\log a}\right) t^{-\frac{2k\pi i}{\log a}} \right) + O(t^m),
$$
for any $m\geq 1$, where the implied constant is uniform in $t$.
Repeating the steps in the proof of Theorem \ref{thm: main}, we deduce that the Hurwitz-type zeta function \eqref{eq: zeta Q=0}, for $z\in\mathbb{C}\setminus(-\infty,0]$ possesses meromorphic continuation to the whole complex $s-$plane with simple poles at $s=-\ell$ and $s=-\ell + \frac{2k\pi i}{\log a}$, $\ell\in\mathbb{N}_0$, $k\in\mathbb{Z}\setminus\{0\}$ and corresponding residues
\begin{eqnarray*}
\underset{%
\begin{array}{c}
s=-\ell
\end{array}%
}{\mathrm{Res}}\zeta _{U}\left( s,z\right) =\frac{z^{\ell}}{\log a}\quad \text{     and   }
\underset{%
\begin{array}{c}
s=\frac{2k\pi i}{\log a}-\ell
\end{array}%
}{\mathrm{Res}}\zeta _{U}\left( s,z\right) =\frac{\left( -z\right)
^{\ell}}{\ell! \log a}\prod%
\limits_{j=1}^{\ell}\left( \frac{2k\pi i}{\log a} -j\right).
\end{eqnarray*}
\end{remark}

The method developed in this paper to deduce meromorphic continuation of the Hurwitz-type zeta function associated to the sequence $\{U_n\}_{n\geq 0}$ which satisfies the recurrence relation of the second order can be adopted to more general recurrence sequences. For example, when the sequence is given by the recurrence relation of order three, with characteristic polynomial having three distinct roots, elements of the sequence can be represented in terms of the Binet formula and the corresponding zeta function can be meromorphically continued using the Taylor series expansion. An example of such sequence is the Tribonacci sequence; some properties of the associated zeta function were studied in \cite{KomatsuTrib}, but mainly for $s=1$. See also \cite{kilic2, WuZh, WuZh2} for some properties of sums of reciprocals of higher order recurrences.

In the follow up paper we plan to investigate zeta functions and the Hurwitz-type zeta functions associated to the sequences satisfying higher order difference equations.

\end{document}